\documentclass[a4paper,11pt]{article}

\usepackage{amsmath,amssymb,amsthm,mathrsfs}
\usepackage[utf8]{inputenc}
\usepackage{a4}
\usepackage{color}
\usepackage{graphicx}
\usepackage[english]{babel}
\usepackage{enumerate}
\usepackage[
  pdfencoding=unicode,
  bookmarks=true,
  bookmarksopen=true,
  bookmarksnumbered=true,
  pdfauthor={Dominic Descombes},
  pdfsubject={Note},
  pdfstartview=FitH,
  colorlinks=false,
]{hyperref}
\hypersetup{
  pdftitle={A characterization of injective subsets in ℓⁿ͚}
}

\numberwithin{equation}{section}

\newtheorem{Thm}{Theorem}[section]
\newtheorem{Prop}[Thm]{Proposition}
\newtheorem{Lem}[Thm]{Lemma}
\theoremstyle{definition}

\newcommand{\R}{\mathbb{R}}

\newcommand{\E}{{\rm E}}

\newcommand{\Del}{\Delta}
\newcommand{\eps}{\varepsilon}

\newcommand{\lam}{\lambda}
\renewcommand{\phi}{\varphi}
\renewcommand{\rho}{\varrho}

\title{A characterization of injective subsets in $\ell_\infty^n$}

\author{Dominic Descombes}

\date{October 14, 2015}

\begin{document}


\pdfbookmark[0]{\texorpdfstring{A characterization of injective subsets in \lowercase{$\ell^n_\infty$}}{A characterization of injective subsets in ℓⁿ͚}}{titleLabel}  
\maketitle

\begin{abstract}
We characterize all (absolute) $1$-Lipschitz retracts $Q$ of $\R^n$ with the maximum norm (denoted $\ell_\infty^n$). They coincide with the subsets written as
\begin{equation*}
Q=\{ x\in\ell_\infty^n \,\, | \,\, \underline{r\!}\,_i(\widehat{x}_i) \leq x_i
\leq
\bar{r}_i(\widehat{x}_i) \text{ for } i=1,\ldots,n \} 
\end{equation*}
where $\widehat{x}_i=(x_1,\!...,x_{i-1},x_{i+1},\!...,x_n)$ and $\underline{r\!}\,_i, \bar{r}_i :
\ell_\infty^{n-1} \rightarrow \R$ are $1$-Lip\-schitz maps with
$\underline{r\!}\,_i \leq\bar{r}_i $ everywhere (provided $Q\neq\emptyset$ and one is allowed to drop any inequality of $Q$).
These sets are also exactly the injective subsets; meaning those $Q\subset\ell_\infty^n$ such that every $1$-Lipschitz map $A\to Q$, defined on a subset $A$ of a metric space $B$, possesses a $1$-Lipschitz extension $B\to Q$.
\end{abstract}
 

\section{Preliminaries}

First we swiftly develop some basics about injective metric spaces needed later on. We adopt the notation of \cite{Lan} and refer thereon and to the original papers \cite{AroP,Isb,Dre} for a comprehensive picture and further results. Then we prove Proposition~\ref{Prop:CharInjSub} mentioned in the abstract.

Throughout this note, $\|\cdot\|$ refers to the supremum norm, and we set $\ell_\infty^n=(\R^n,\|\cdot\|)$ and $\ell_\infty(X)=(\ell_\infty(X),\|\cdot\|)$ where the latter space is composed of all real valued and bounded functions on an arbitrary index set $X$.
A metric space $X$ is {\em injective} if for every metric space $B$ and every 
$1$-Lipschitz map $\Phi \colon A \to X$ defined on a subet $A \subset B$ there exists a 
$1$-Lipschitz extension $\overline{\Phi} \colon B \to X$. Basic examples include the real line ($\overline{\Phi}(b):=\inf_{a\in A}(\Phi(a)+d(a,b))$ yields a maximal $1$-Lipschitz extension) and all $\ell_\infty$ spaces for arbitrary index sets (as the maps may be extended component-wise) and ($\R$-)trees. Also, every $1$-Lipschitz retract $X\subset Y$ of an injective space $Y$ is itself injective since $\Phi\colon A\to X$ may first be extended to $\overline{\Phi}\colon B\to Y$ and then concatenated with the retraction $\pi\colon Y\to X$. A metric space $X$ is called an {\em absolute $1$-Lipschitz retract} if for every isometric embedding $i\colon X\to Y$ into another metric space $Y$ there exists a retraction of $Y$ onto $i(X)$. Clearly, every injective space $X$ is an absolute $1$-Lipschitz retract since the identity on $Y$ restricted to $i(X)$ extends to a $1$-Lipschitz retraction. Conversely, every space $X$ embeds isometrically into $\ell_\infty(X)$ via the Kuratowski embedding $x\mapsto d_x-d_z$, where $d_x$ denotes the map $y\mapsto d(x,y)$, and is therefore injective if it is a $1$-Lipschitz retract. Thus these two notions coincide, and injective spaces are geodesic, contractible, and also never empty (as we allow $A=\emptyset\neq B$ in the definition above).

By a remarkable theorem of Isbell~\cite{Isb}, every metric 
space $X$ possesses an {\em injective hull} $\E(X)$ (unique up to isometry) which is minimal among injective spaces containing an isometric copy of $X$ in the following sense.
There is an isometric embedding ${\rm e}\colon X\to\E(X)$ with the following property. Whenever there is a metric space $Z$ and a $1$-Lipschitz map $h\colon\E(X)\to Z$ such that $h\circ {\rm e}$ is an isometric embedding, then $h$ is an isometric embedding as well.
However, we do not need this definition in the sequel and now 
review just as much of Isbell's construction $X\mapsto\E(X)$ as necessary for our purposes.
Given a metric space $X$, we denote by $\R^X$ the vector space of arbitrary 
real valued functions on $X$ and put
\[
\Del(X) := \{ f \in \R^X \,|\, 
\text{$f(x) + f(y) \ge d(x,y)$ for all $x,y \in X$}\}.
\]
A function $f \in \Del(X)$ is called {\em extremal\/} if there is no $g \le f$ in $\Del(X)$ 
other than $f$. The set $\E(X)$ can equivalently be written as
\[
\E(X) = \bigl\{ f \in \R^X \,|\, 
\text{$f(x) = \textstyle\sup_{y \in X}(d(x,y) - f(y))$ for all $x \in X$} 
\bigr\},
\]
thus $f$ is extremal if and only if for every $x$ and $\eps>0$ there is a $y$ such that
\begin{equation}\label{eq:extremal}
f(x)+f(y)\leq d(x,y)+\eps .
\end{equation}
Applying the equation defining the members of $\E(X)$ twice, we obtain
\[
f(x)-d(x,x')=\sup_{y\in X}(d(x,y)-d(x,x')-f(y))\leq f(x')
\]
for all $x,x'\in X$, so every extremal $f$ is $1$-Lipschitz. From the definition of $\Delta(X)$ we see that every function therein is non-negative everywhere. Thereof it follows easily that all the functions $d_x$ for $x\in X$ are extremal and, moreover, are the only extremal functions with zeros. 
Isbell then goes on to show that the set of extremal functions is in fact an injective hull of $X$. We will not demonstrate that here as we only need the following fact used in Lemma~\ref{Lem:A4}. If $X$ is injective, then the only extremal functions are $\{d_x\,|\,x\in X\}$. In order to show this, assume the existence of an extremal $f$ that is not equal to any $d_x$ for $x\in X$, in particular $f$ has no zero. Now let $\tilde X$ be the disjoint union of $X$ and a single point $x_f$. The metric $\tilde d$ on $\tilde X$ shall be $d(x,y)$ for $x,y\in X$ and $f(x)$ for the distance from $x$ to $x_f$ --- and the properties of a metric hold as $f\in\Delta(X)$ and $f$ is $1$-Lipschitz and positive. Next let $\overline{\Phi}\colon \tilde{X}\to X$ be a $1$-Lipschitz extension of the identity on $X$. So
\[
d(\overline{\Phi}(x_f),x) = d(\overline{\Phi}(x_f),\overline{\Phi}(x))\leq d(x_f,x)=f(x)
\]
for all $x\in X$. Hence $d_{\overline{\Phi}(x_f)}\leq f$, thus $f=d_{\overline{\Phi}(x_f)}$, a contradiction.


\section{The characterization}

For every element $x$ in $\ell_\infty^n$ let $\widehat{x}_i := (x_1,\ldots,x_{i-1},x_{i+1},\ldots,x_n)
\in \ell_\infty^{n-1}$ denote the vector $x$ but with the $i$-th coordinate
omitted. In the following we are using the $1$-Lipschitz map $(a,b,x)\mapsto\min\{\max\{a,x\},b\}$ from
$\ell_\infty^3$ to $\R$ quite frequently where $a\in\{-\infty\}\cup\R$, $b\in\R\cup\{\infty\}$ but we always assert that $a\leq b$. We write $p([a,b],x)$ for the value of this projection, i.e.\ the unique real number in the closed interval $\{y\in\R\,|\, a\leq y\leq b\}$ closest to $x$.

\begin{Lem}\label{Lem:A1} Let $\underline{r\!},
\bar{r} :
\ell_\infty^{n-1} \rightarrow \R$ be two $1$-Lipschitz maps with
$\underline{r\!}\leq\bar{r}$ (at every point in $\ell_\infty^{n-1}$).
Then the map
\[
\varphi:\ell_\infty^n \rightarrow \ell_\infty^n, \; \varphi(x):= \bigl( 
p([\underline{r\!}\,(\widehat{x}_1),\bar{r}(\widehat{x}_1)],x_1),x_2,\ldots,x_n \bigr)
\]
is a $1$-Lipschitz retraction to the hence
injective and non-empty subset
\[
Q := \left\{ x\in\ell_\infty^n \, | \, \underline{r\!}\,(\widehat{x}_1) \leq x_1
\leq
\bar{r}(\widehat{x}_1) \right\}
\]
with the property that $\widehat{\varphi(x)}_1 = \widehat{x}_1$. The lemma remains valid if we allow the (constant) function $\underline{r\!} = -\infty$ as the lower or $\bar{r} = \infty$ as the upper bound or both.
\end{Lem}

\begin{proof} Every component of $\varphi$ is $1$-Lipschitz, hence the
whole map is and the rest is simple to prove as well.
\end{proof}

\begin{Lem}\label{Lem:A2} Let $\lambda\in [0,1)$ and for every
$i=1,\ldots,n$ let $\underline{r\!}\,_i, \bar{r}_i :
\ell_\infty^{n-1} \rightarrow \R$ be a pair of $\lambda$-Lipschitz maps with
$\underline{r\!}\,_i \leq\bar{r}_i $.
Then the subspace
\[
Q := \{ x\in\ell_\infty^n \, | \, \underline{r\!}\,_i(\widehat{x}_i) \leq x_i
\leq
\bar{r}_i(\widehat{x}_i) \text{ for } i=1,\ldots,n \}
\]
is injective --- in particular every such set of
inequalities is solvable. The lemma remains valid if any of the lower or upper bounds take the constant value $-\infty$ or $\infty$, respectively.
\end{Lem}

\begin{proof} Denote by $\varphi_1$ the map from the previous lemma
(applied to $\underline{r\!}\,_1,\bar{r}_1$) and analogously
$\varphi_2,\ldots\varphi_n$ for the other components. We define the following
sequence of $1$-Lipschitz maps
\[
\Phi_m :=
\varphi_k\circ\cdots\circ\varphi_2\circ\varphi_1\circ(\varphi_n\circ\cdots\circ\varphi_2\circ\varphi_1)\circ\cdots\circ(\varphi_n\circ\cdots\circ\varphi_2\circ\varphi_1) ,
\]
where $m$ denotes the number of concatenated maps on the right hand side. Formally $\Phi_0=\text{Id}_{\ell_\infty^n}$ and
$\Phi_{m}=\varphi_{[m]}\circ\Phi_{m-1}$, where
$[m]:=((m-1)\text{ mod }n)+1$.\newline We claim that the sequence $\Phi_m$
converges pointwise and for every $x\in\ell_\infty^n$ we have $\lim_{m\to\infty}\Phi_m(x) \in Q$. Then, since every ($1$-Lipschitz) $\Phi_m$ fixes $Q$, the pointwise limit $\Phi$ is a
$1$-Lipschitz retraction to $Q$ concluding the proof.\newline
In order to prove the claim, pick an arbitrary $x\in\ell_\infty^n$ and define $d_k$
to be a sequence of real numbers such that
\[
\Phi_m(x)=x+\sum_{k=1}^m d_k e_{[k]} \;\;\text{and consequently}\;\; |d_m| =
\|\Phi_m(x)-\Phi_{m-1}(x)\| ,
\]
where $e_{[k]}$ denotes the $[k]$-th vector of the standard basis. This
is possible since $\Phi_m(x)-\Phi_{m-1}(x) =
\varphi_{[m]}(\Phi_{m-1}(x))-\Phi_{m-1}(x)$ and the map $\varphi_{[m]}$ alters
nothing but the $[m]$-th coordinate. Now let $m>n$ and set
$i:=[m]$ as well as 
\[
a=\Phi_{m-n}(x), \; b=\Phi_{m-1}(x), \; c=\Phi_m(x) .
\]
We want an estimate for $|d_m|=\|c-b\|$. Since
$a=\varphi_i(\Phi_{m-n-1}(x))$, we have $\underline{r\!}\,_i(\widehat{a}_i)\leq
a_i\leq\bar{r}_i(\widehat{a}_i)$ and also
$b=a+\sum_{k=m-n+1}^{m-1} d_ke_{[k]}$, hence $b_i=a_i$ and
$\|b-a\|=\max\{|d_{m-n+1}|,|d_{m-n+2}|,\ldots,|d_{m-1}|\}$. We compare
the following two lines:
\begin{align*}
|d_m| & = |p([\underline{r\!}\,_i(\widehat{b}_i),\bar{r}_i(\widehat{b}_i)],b_i) - b_i| = |c_i-b_i| , \\
0 & = |p([\underline{r\!}\,_i(\widehat{a}_i),\bar{r}_i(\widehat{a}_i)],a_i) - a_i| .
\end{align*}
As $(p,q)\mapsto|\min\{p, \max\{ a_i,q\} \} - a_i|$ is a $1$-Lipschitz map
from $\ell_\infty^2$ to $\R$ and $\left\|\left(\underline{r\!}\,_i(\widehat{b}_i)-\underline{r\!}\,_i(\widehat{a}_i),\bar{r}_i(\widehat{b}_i)-\bar{r}_i(\widehat{a}_i)\right)\right\|\leq\lam\|b-a\|$, because $\underline{r\!}\,_i,\bar{r}_i$ are $\lam$-Lipschitz by assumption, this yields
\begin{equation}\label{Eq:LemA3} |d_m| \leq \lambda\|b-a\| =
\lambda\max\{|d_{m-n+1}|,|d_{m-n+2}|,\ldots,|d_{m-1}|\} .
\end{equation}
To end the proof we set $D:=\max\{|d_1|,\ldots,|d_n|\}$ and get
\[ |d_m|\leq D \lambda^{(m-1) \% n} \]
(where $ \% $ denotes integer division, the floored value of
the real division) by induction using \eqref{Eq:LemA3}. So the convergence of $\Phi_m(x)$
is like that of a geometric series.
It remains to check that $\lim_{m\to\infty} \Phi_m(x) \in Q$. This
follows immediately from the fact that the sets $\varphi_i(\ell_\infty^n)
= \{ x\in\ell_\infty^n \, | \, \underline{r\!}\,_i(\widehat{x}_i) \leq x_i
\leq
\bar{r}_i(\widehat{x}_i)\}$ are closed and the subsequences
$k\mapsto\Phi_{i+kn}(x)$ (convergent to the same limit as $\Phi_m(x)$) lie completely in
these sets. So the limit lies in the intersection of all the sets
$\varphi_i(\ell_\infty^n)$ and this is exactly $Q$.
\end{proof}

The proof does not work with $\lambda=1$ since the
sequence $\Phi_m(x)$ need not be convergent as the example $n=2,\;
\underline{r\!}\,_1(x_2)=\bar{r}_1(x_2)=x_2,\; \underline{r\!}\,_2(x_1)=\bar{r}_2(x_1)=1+x_1$
shows.
Moreover, $Q$ can be empty. But even assuming $Q\neq\emptyset$ (change
$1+x_1$ to $-x_1$ in the example) the sequence can be divergent, and it is of no
help to subtract a convergent subsequence (which, in the case $Q\neq\emptyset$,
always exists).
Nevertheless, we can show the lemma to hold for
$\lambda\leq 1$ assuming $Q\neq\emptyset$.

Recall that for two non-empty subsets $A,B$ of some metric space $X$ the {\em Hausdorff distance} is the real value
\[
d_H(A,B):=\inf\{\eps \,|\, A\subset U_\eps(B), B\subset U_\eps(A) \} ,
\]
where $U_\eps(A)$ denotes the closed $\eps$-neighborhood of $A$. And this is a metric on the set of all closed, bounded, and non-empty subsets of a given metric space $X$.
For further information on this distance and the Gromov-Hausdorff convergence, we refer to \cite{BriH}.

\begin{Lem}\label{Lem:A3} For every
$i=1,\ldots,n$ let $\underline{r\!}\,_i, \bar{r}_i :
\ell_\infty^{n-1} \rightarrow \R$ be a pair of $1$-Lipschitz maps with
$\underline{r\!}\,_i \leq\bar{r}_i $.
Then the subspace
\[
Q := \{ x\in\ell_\infty^n \, | \, \underline{r\!}\,_i(\widehat{x}_i) \leq x_i
\leq
\bar{r}_i(\widehat{x}_i) \text{ for } i=1,\ldots,n \}
\]
is injective assuming either that all the maps $\underline{r\!}\,_i,
\bar{r}_i$ are bounded (hence again the system of inequalities is solvable
automatically) or that $Q$ is non-empty (requiring the existence of a
solution). Again, the lemma remains valid if any of the lower or upper bounds take the constant value $-\infty$ or $\infty$, respectively.
\end{Lem}

\begin{proof} The bounded case first. Let $\lambda_k=1-1/k$ and $l$ be a lower
bound for the maps $\underline{r\!}\,_i$ and $u$ an upper bound for the maps
$\bar{r}_i$. Define $\bar{r}_i^k=\lambda_k(\bar{r}_i-u)+u$ and
$\underline{r\!}\,_i^k=\lambda_k(\underline{r\!}\,_i-l)+l$ and observe that the maps
$\underline{r\!}\,_i^k, \bar{r}_i^k$ for some fixed $k$ are all
$\lambda_k$-Lipschitz. Moreover, for every fixed $i$ the sequence
$\bar{r}_i^k$ ($\underline{r\!}\,_i^k$) is monotonically converging down (up) to $\bar{r}_i$ ($\underline{r\!}\,_i$) pointwise. Define the sets 
\[
Q_k := \left\{
x\in\ell_\infty^n \, \Big| \, \underline{r\!}\,_i^k(\widehat{x}_i) \leq x_i
\leq \bar{r}_i^k(\widehat{x}_i) \text{ for }
i=1,\ldots,n \right\}
\]
which are injective by the first proposition. And we
have $Q_1\supset Q_2\supset Q_3 \supset \cdots$ as well as $\bigcap_k Q_k = Q$
(this already implies $Q\neq\emptyset$ as all the $Q_k$ are compact). If we can
show that $Q_k$ converges to $Q$ w.r.t.\ Hausdorff distance (implying
Gromov-Hausdorff convergence), then $Q$ is injective --- either see Section~1.5 in \cite{Moe} or derive a direct
proof of this easy special case here. So assume for some $\epsilon>0$
we would have $d_\text{H}(Q_k,Q)>\epsilon$ for every $k$ (first only for infinitely
many $k$ but then for all by monotonicity of $Q_k$). Then taking a convergent sequence
$x_k\in Q_k\setminus U_\epsilon(Q)\neq\emptyset$ one immediately gets
$\lim_{k\to\infty} x_k\in \bigcap Q_k$ leading to a contradiction. (In fact,
this could also be derived from well-known theorems about Hausdorff distance.)
This ends the proof in the case of bounded maps

To reduce the case assuming $Q\neq\emptyset$ (with possibly unbounded $Q$) to the previous one we refer to Corollary~1.19 in \cite{Moe}. There it is shown that every proper metric space $Q$ is injective if and only if every closed ball in $Q$ is injective. Thereby we only need to verify that $Q\cap B(q,r)$ is injective for all $q\in Q$, may assume $q=0\in Q$ and have
\begin{align*}
 Q\cap B(0,r) & = \left\{ x \, | \,
 p([-r,\infty],\underline{r\!}\,_i(\widehat{x}_i)) \leq x_i
\leq p([-\infty,r],\bar{r}_i(\widehat{x}_i)) \text{ for all } i \right\}\\
 & = \left\{ x \, | \,
 p([-r,r],\underline{r\!}\,_i(\widehat{x}_i)) \leq x_i
\leq p([-r,r],\bar{r}_i(\widehat{x}_i)) \text{ for all } i \right\} .
\end{align*}
The last set above is injective by the bounded case, and we are left to show the second equality since the first one is clear. Clearly, all three sets are contained in $B(0,r)$. For a point $x$ in that ball, we have $\|\widehat{q}_i-\widehat{x}_i\|\leq r$ and $\underline{r\!}\,_i(\widehat{q}_i)\leq 0\leq \bar{r}_i(\widehat{q}_i)$ by our choice of $q$. Consequently, $\underline{r\!}\,_i(\widehat{x}_i)\leq r$, $-r\leq\bar{r}_i(\widehat{x}_i)$ and from this it is now obvious that the sets coincide.
\end{proof}

Now we turn to the converse that every injective subset of $\ell_\infty^n$ can be written as a set of $2n$ inequalities in the sense of the lemma above. We need a last definition before turning to the proof.
By a {\em cone} $C(p,x)$ in a metric space $X$ we mean the set
\[
\{ q\in X \,|\, d(p,q) = d(p,x)+d(x,q) \} .
\]
For cones of the form $C(x-e_i,x)$ or $C(x+e_i,x)$, where $e_i, i=1,\ldots,n$ are vectors of the standard basis of $\R^n$, we shorten the notation further to $C(x,+i):=C(x-e_i,x)$ and $C(x,-i):=C(x+e_i,x)$. With $\bot$ being the Euclidean orthogonality relation, we can equivalently write
\[
C(x,+i) = \{ x + be_i+y \,|\, b\geq 0, y\bot e_i, \|y\|\leq b \}.
\]

\begin{Lem}\label{Lem:A4} For every injective subset $Q\subset\ell_\infty^n$ there are
$n$ pairs of $1$-Lipschitz maps $\underline{r\!}\,_i, \bar{r}_i :
\ell_\infty^{n-1} \rightarrow \R$ ($i=1,\ldots,n$) with
$\underline{r\!}\,_i \leq\bar{r}_i $ such that
\[
Q = \{ x\in\ell_\infty^n \, | \, \underline{r\!}\,_i(\widehat{x}_i) \leq x_i
\leq
\bar{r}_i(\widehat{x}_i) \text{ for } i=1,\ldots,n \} .
\]
We allow for $\underline{r\!}\,_i=-\infty$ or $\bar{r}_i=\infty$ (or both) for any $i$ which is equivalent to drop some of the inequalities.
\end{Lem}

\begin{proof} The injective subsets of $\ell_\infty^1=\R$ are exactly the closed intervals. Thereby the case $n=1$ is trivial or more a matter of declaring the convention $\widehat{x}_i=0\in\ell_\infty^0$, and so we assume $n\geq 2$. Since $Q$ is injective, only the distance functions $d_q$ for $q\in Q$ are extremal (as shown in the preliminaries). Therefore, given any point $x$ outside $Q$, the function $d_x|_Q$ is non-extremal. So we may assign to every such point the positive quantity
\begin{align*}
\eps & (x) := \\
 &\sup\{\eps \,|\, \text{there is }p\in Q \text{ with } \|x-p\|+\|x-q\|\geq \|p-q\|+\eps \text{ for all } q\in Q \} 
\end{align*}
(from choosing $q\in Q$ such that $\|x-q\|=d(x,Q)$ we see that $\eps(x)\leq 2 d(x,Q)$).
Moreover, for every $x$ let $p_x$ be such that $\|x-p_x\|+\|x-q\|\geq \|p_x-q\|+\eps(x)/2$ for every $q\in Q$. Next we select a cone $C_x$ for every $x\in\ell_\infty^n\setminus Q$. To that end, let $a\in\R$ be some positive number, the exact value will be determined in the course of the proof. Assume that the $i$-th coordinate of $x-p_x$ has maximal absolute value among all coordinates (if there are several such coordinates we simply choose one). Now set $C_x=C(x-a\eps(x)e_i,+i)$ if that value is positive and $C_x=C(x+a\eps(x)e_i,-i)$ if it is negative. Observe that $x\in\operatorname{Interior}(C_x)$ always. Assume that $Q\cap C_x$ contains a point $q$ and $C_x:=C(x-a\eps(x)e_i,+i)$ (the case $C_x:=C(x+a\eps(x)e_i,-i)$ works the same way). A straightforward computation yields $C(x,+i)\subseteq C(p_x,x)$ and hence
\[
\|p_x-(q+a\eps(x)e_i)\| = \|p_x-x\|+\|x-(q+a\eps(x)e_i)\|
\]
and consequently
\[
\|p_x-q\|\geq \|p_x-x\|+\|x-q\|-2a\eps(x).
\]
But this violates the definition of $p_x$ if we choose $a< 1/4$. We do so and have $Q\cap C_x = \emptyset$ for all $x\in\ell_\infty^n\setminus Q$. For every $i$, we define $\bar{r}_i$ to be the pointwise infimum over the family of $1$-Lipschitz functions $\ell_\infty^{n-1}\to\R; y\mapsto\|\widehat{x}_i-y\|+x_i-a\eps(x)$ where every $x$ such that $C_x=C(x-a\eps(x)e_i,+i)$ contributes exactly one member. If there is no such $x$, we let $\bar{r}_i:=\infty$. Similarly, $\underline{r\!}\,_i:=-\infty$ if there is no $x$ with $C_x=C(x+a\eps(x)e_i,-i)$ or otherwise the supremum over all functions $y\mapsto\|\widehat{x}_i-y\|+x_i+a\eps(x)$ for $x$ with $C_x=C(x+a\eps(x)e_i,-i)$. It is now not hard to deduce
\[
Q = \{ x\in\ell_\infty^n \, | \, \underline{r\!}\,_i(\widehat{x}_i) \leq x_i
\leq
\bar{r}_i(\widehat{x}_i) \text{ for } i=1,\ldots,n \}
\]
from $x\in\operatorname{Interior}(C_x)$ and $Q\cap C_x = \emptyset$. It remains to show $\underline{r\!}\,_i\leq\bar{r}_i$ for all pairs. First notice that $\underline{r\!}\,_i>\bar{r}_i$ at some point in $\ell_\infty^{n-1}$ implies there are points $x,y\in\ell_\infty^n\setminus Q$ with $C_x:=C(x-a\eps(x)e_i,+i)$, $C_y:=C(y+a\eps(y)e_i,-i)$ such that the intersection $\operatorname{Interior}(C_x)\cap\operatorname{Interior}(C_y)$ is not empty. To show that this can not happen for appropriate choice of $a$, we assume otherwise and start with the easy observation that the apex $x-a\eps(x)e_i$ of $C_x$ lies in $\operatorname{Interior}(C_y)$. Therefore $\tilde{x}:=x-a\eps(x)e_i-a\eps(y)e_i$ lies in $\operatorname{Interior}(C(y,-i))$ and the same holds for $\tilde{p_x}:=p_x-a\eps(x)e_i-a\eps(y)e_i$ as $p_x\in C(x,-i)$.
\begin{figure}[ht!]
\def\svgwidth{0.71\textwidth}
\vspace{0.1cm}
\hspace{1.8cm}
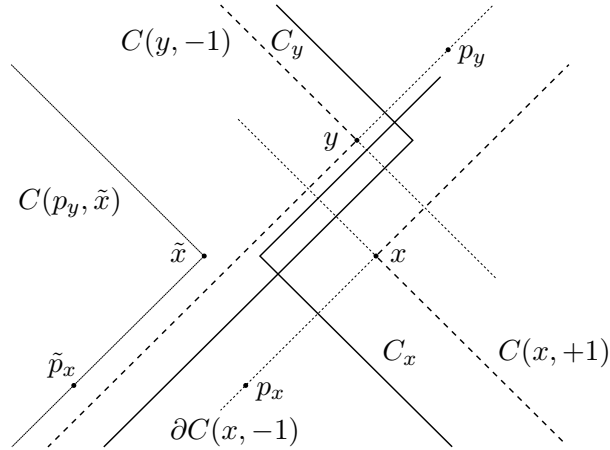
\caption{In this illustration $p_x,p_y$ are chosen in $\partial C(x,-1),\,\partial C(y,+1)$, respectively, since these are the critical cases.}
\end{figure}

\noindent So we have 
\begin{align*}
\|x-p_x\|+\|x-p_y\|&\leq \|\tilde{x}-\tilde{p_x}\|+\|\tilde{x}-p_y\|+a(\eps(x)+\eps(y)) \\
&= \|\tilde{p_x}-p_y\|+a(\eps(x)+\eps(y)) \\
& \leq \|p_x-p_y\|+2a(\eps(x)+\eps(y)),
\end{align*}
hence by definition of $p_x$ and $p_y$ this leads to $\eps(x)\leq 4a(\eps(x)+\eps(y))$ and $\eps(y)\leq 4a(\eps(x)+\eps(y))$, respectively. Now take $a<1/8$. The sum of the last two inequalities involving $\eps(x),\eps(y)$ then yields a contradiction proving that $\operatorname{Interior}(C_x)\cap\operatorname{Interior}(C_y)$ is in fact empty.
\end{proof}

All in all we arrive at the final proposition which summarizes the previous four lemmas.

\begin{Prop}\label{Prop:CharInjSub} A non-empty subset $Q\subset\ell_\infty^n$ is injective if and only if it can be written as a solution set of (at most $2n$) inequalities like 
\[
Q = \{ x\in\ell_\infty^n \, | \, \underline{r\!}\,_i(\widehat{x}_i) \leq x_i
\leq
\bar{r}_i(\widehat{x}_i) \text{ for } i=1,\ldots,n \} 
\]
where $\underline{r\!}\,_i, \bar{r}_i :
\ell_\infty^{n-1} \rightarrow \R$ ($i=1,\ldots,n$) are $1$-Lipschitz maps with
$\underline{r\!}\,_i \leq\bar{r}_i $ 
and one is allowed to drop any of these inequalities.
\end{Prop}


\bigskip\noindent
D.~Descombes ({\tt dominic.descombes@math.ethz.ch}),\\
Department of Mathematics, ETH Zurich, 8092 Zurich, Switzerland 


\end{document}